\documentclass{amsproc}

\usepackage{mathtools}
\usepackage{tikz}
\usepackage{tikz-3dplot}
\usepackage{hyperref}
\usepackage{mathrsfs}
\usepackage{courier}
\usepackage{comment}
\usepackage{graphicx}

\def\J{\mathcal{J}}
\def\O{\mathcal{O}}
\def\P{\mathbb{P}}
\def\R{\mathbb{R}}
\def\Sp{\mathbb{S}}
\def\V{\mathcal{V}}

\def\1{\mathbf{1}}

\def\dP{d_{\P^2}}
\def\dS{d_{\Sp^2}}

\def\mat#1{\begin{bmatrix}#1\end{bmatrix}}

\def\vspan{\text{span}}

\def\LRarrow{\Leftrightarrow}

\theoremstyle{plain}
\newtheorem{theorem}{Theorem}
\newtheorem{lemma}{Lemma}

\newtheorem{definition}[lemma]{Definition}

\newtheorem{construction}[lemma]{Construction}
\newtheorem{proposition}[lemma]{Proposition}
\newtheorem{problem}{Problem}

\newtheorem{remark}{Remark}

\numberwithin{equation}{section}

\makeatletter
\DeclareFontFamily{U}{tipa}{}
\DeclareFontShape{U}{tipa}{m}{n}{<->tipa10}{}
\newcommand{\Arc@char}{{\usefont{U}{tipa}{m}{n}\symbol{62}}}%

\newcommand{\Arc}[1]{\mathpalette\Arc@Arc{#1}}

\newcommand{\Arc@Arc}[2]{%
  \sbox0{$\m@th#1#2$}%
  \vbox{
    \hbox{\resizebox{\wd0}{\height}{\Arc@char}}
    \nointerlineskip
    \box0
  }%
}
\makeatother

\begin{document}

\title{The Fermat-Torricelli Problem in the Projective Plane}


\author{Manolis C. Tsakiris}
\address{Key Laboratory of Mathematics Mechanization, Academy of Mathematics and Systems Science, Chinese Academy of Sciences, Beijing, 100190, China}
\email{manolis@amss.ac.cn}

\address{School of Information Science and Technology, ShanghaiTech University, No.393 Huaxia Middle Road, Pudong, Shanghai, China}
\email{mtsakiris@shanghaitech.edu.cn}

\author{Sihang Xu }
\address{School of Information Science and Technology, ShanghaiTech University, No.393 Huaxia Middle Road, Pudong, Shanghai, China}
\email{xush@shanghaitech.edu.cn}

\begin{abstract}
We pose and study the Fermat-Torricelli problem for a triangle in the projective plane under the sine distance. Our main finding is that if every side of the triangle has length greater than $\sin 60^\circ$, then the Fermat-Torricelli point is the vertex opposite the longest side. Our proof relies on a complete characterization of the equilateral case together with a deformation argument.  
\end{abstract}

\maketitle

\section{Introduction}

The classical Fermat-Torricelli problem is concerned with the location of the point on the plane that minimizes the sum of its Euclidean distances to three given non-collinear points. A complete answer was already prescribed by Torricelli   \cite{vivianimaximis} and by now many different proofs exist, e.g., see \cite{eriksson1997fermat}. Fermat-Torricelli problems have been extended from planar to spherical triangles \cite{MGP,cockayne1967steiner,cockayne1972fermat}, inner product spaces \cite{benitez2002location,benitez2007middle}, Minkowski spaces \cite{chakerian1985fermat,martini2002fermat}, Riemannian manifolds \cite{yang2010riemannian,afsari2011riemannian} or even to a tropical geometry setting \cite{lin2018tropical}. Of great practical significance is the generalization of Fermat-Torricelli points to so-called Steiner-Weber points, where one seeks the point that minimizes the sum of the weighted distances to $n$ other points, e.g., see \cite{kupitz1997geometric, zachos2008weighted, hajja2017complete}. There are three main questions attached to each such problem instance: existence, uniqueness and location \cite{benitez2002location}.  These depend in a complicated way on the properties of the underlying manifold and, with the exception of the planar case (e.g., \cite{barbara2000fermat,uteshev2014analytical}), they are in principle difficult to answer. A notable step forward in this direction is the work of Afsari \cite{afsari2011riemannian}. 

The Fermat-Torricelli point can be thought of as a geometric median, as opposed to the geometric mean. This latter relies on minimizing the sum of the squared distances and, as such, it is much easier to deal with analytically and computationally. On the other hand, geometric medians are known to be robust to outliers and thus are very important for contemporary applications in statistics and adjacent fields. Recently, in the machine learning problem of hyperplane clustering, where one is concerned with clustering a given set of points $\mathscr{X} \subset \mathbb{R}^D$ sampled from an unknown arrangement of $n$ hyperplanes $\mathcal{H}_1,\dots,\mathcal{H}_n$ of $\mathbb{R}^D$, Tsakiris \& Vidal showed in \cite{tsakiris2017hyperplane} that fitting a hyperplane $\mathcal{H}$ to the points $\mathscr{X}$ so that the sum of the distances of the points to $\mathcal{H}$ is minimized, can be seen as a discrete version of a Steiner-Weber problem in the projective space $\mathbb{P}^{D-1}$ under the sine distance. Inspired by that work, we are here concerned with the following special case:

\begin{problem}[Fermat-Torricelli points of a triangle in the projective plane] \label{problem:Fermat-Torricelli}
Let $\P^2$ be the real projective plane. For $P_1,P_2\in\P^{2}$ let $\varphi_{P_1P_2}$ be the smallest angle between the lines $\ell_{P_1}$ and $\ell_{P_2}$ through the origin in $\mathbb{R}^3$, represented by $P_1$ and $P_2$ respectively, and set $d_{\P^{2}}(P_1,P_2)=\sin\varphi_{P_1P_2}$.
Given non-collinear points $A,B,C \in \mathbb{P}^{2}$, characterize the solutions $\mathscr{P}_{ABC} \subset \P^2$ to the problem
$$ \min_{P \in \mathbb{P}^{2}} \, \, \, \J_{ABC}(P)=\dP(A,P)+\dP(B,P)+\dP(C,P). $$ 
\end{problem}

Within the Fermat-Torricelli and Steiner-Weber literature, Problem \ref{problem:Fermat-Torricelli} represents a natural progression from planar to spherical to projective triangles. The main contribution of this paper is a remarkable property of the Fermat-Torricelli point of a triangle in the projective plane, as defined in Problem \ref{problem:Fermat-Torricelli}, placing it on one or more of the vertices of the triangle, providing that $\varphi_{AB}, \, \varphi_{AC}, \, \varphi_{BC}$ are all greater than $60^\circ$. In particular, under the said condition, the Fermat-Torricelli point is unique if and only if there is a unique side of the triangle of maximal length. Embedded in our proof is another interesting result, a complete characterization of the equilateral case. More precisely: 

\begin{theorem} \label{thm:theorem}
Let $A,B,C \in \P^2$ be ordered such that $\varphi_{AB}\le\varphi_{AC}\le\varphi_{BC}$. 
\begin{enumerate}
\item If $60^\circ \le \varphi_{AB}\le \varphi_{AC}<\varphi_{BC}$, then $\mathscr{P}_{ABC}=\{A\}$. 
\item If $60^\circ \le \varphi_{AB}< \varphi_{AC}=\varphi_{BC}$, then $\mathscr{P}_{ABC}=\{A,B\}$. 
\item If $\varphi_{AB}=\varphi_{AC}=\varphi_{BC}=\varphi$, denote by $E$ the point of $\P^2$ that represents the centroid $\ell_E$ in $\mathbb{R}^3$ of the three equiangular lines $\ell_A,\ell_B,\ell_C$ associated to $A,B,C$; see Figure \ref{fig:equiangular-case}. We have:
\begin{enumerate}
\item If $\varphi > 60^\circ$, then $\mathscr{P}_{ABC}=\{A,B,C\}$.
\item If $\varphi = 60^\circ$, then $\mathscr{P}_{ABC}=\{A,B,C,E\}$.
\item If $\varphi < 60^\circ$, then $\mathscr{P}_{ABC}=\{E\}$. 
\end{enumerate}
\end{enumerate}
\end{theorem}

\tdplotsetmaincoords{30}{75}

\pgfmathsetmacro{\vlong}{5/sqrt(27)}
\pgfmathsetmacro{\vshort}{1/sqrt(27)}

\def\vmlong{0.8165}
\def\vmshort{0.4082}

\def\vax{\vlong}\def\vay{\vshort}\def\vaz{\vshort}
\def\vby{\vlong}\def\vbx{\vshort}\def\vbz{\vshort}
\def\vcz{\vlong}\def\vcx{\vshort}\def\vcy{\vshort}

\pgfmathsetmacro{\vpx}{1/sqrt(3)}
\pgfmathsetmacro{\vpy}{1/sqrt(3)}
\pgfmathsetmacro{\vpz}{1/sqrt(3)}

\pgfmathsetmacro{\vplong}{\vlong*(\vpx*\vax+\vpy*\vay+\vpz*\vaz)}
\pgfmathsetmacro{\vpshort}{\vshort*(\vpx*\vax+\vpy*\vay+\vpz*\vaz)}
\def\vaex{\vplong}\def\vaey{\vpshort}\def\vaez{\vpshort}
\def\vbey{\vplong}\def\vbex{\vpshort}\def\vbez{\vpshort}
\def\vcez{\vplong}\def\vcex{\vpshort}\def\vcey{\vpshort}

\tdplotsetmaincoords{30}{75}

\pgfmathsetmacro{\vlong}{10/sqrt(27)}
\pgfmathsetmacro{\vshort}{2/sqrt(27)}
\pgfmathsetmacro{\vmid}{2/sqrt(3)}

\def\vmlong{0.8165}
\def\vmshort{0.4082}

\def\vax{\vlong}\def\vay{\vshort}\def\vaz{\vshort}
\def\vby{\vlong}\def\vbx{\vshort}\def\vbz{\vshort}
\def\vcz{\vlong}\def\vcx{\vshort}\def\vcy{\vshort}

\pgfmathsetmacro{\vpx}{1/sqrt(3)}
\pgfmathsetmacro{\vpy}{1/sqrt(3)}
\pgfmathsetmacro{\vpz}{1/sqrt(3)}


\begin{figure}[ht]
    \centering
    \tdplotsetmaincoords{30}{75}

    \def\vex{\vmid}\def\vey{\vmid}\def\vez{\vmid}
    \begin{tikzpicture}[scale=1.6,tdplot_main_coords]

        \coordinate (O) at (0,0,0);
        \draw[dashed] (0,0,0) -- (\vax,\vay,\vaz) node (A)[anchor=north east]{$A$};
        \draw[dashed] (0,0,0) -- (\vbx,\vby,\vbz) node (B)[anchor=south west]{$B$};
        \draw[dashed] (0,0,0) -- (\vcx,\vcy,\vcz) node (C)[anchor=south east]{$C$};
        \draw[dashed] (0,0,0) -- (\vex,\vey,\vez) node (E)[anchor=north west]{$E$};

        \draw[dashed] (\vax,\vay,\vaz) -- (1.2*\vax,1.2*\vay,1.2*\vaz) node[anchor=north west]{$l_A$};
        \draw[dashed] (\vbx,\vby,\vbz) -- (1.2*\vbx,1.2*\vby,1.2*\vbz) node[anchor=west]{$l_B$};
        \draw[dashed] (\vcx,\vcy,\vcz) -- (1.6*\vcx,1.6*\vcy,1.6*\vcz) node[anchor=south]{$l_C$};
        \draw[dashed] (\vex,\vey,\vez) -- (1.6*\vex,1.6*\vey,1.6*\vez) node[anchor=north west]{$l_E$};
        
        \node[draw=none,rotate = 70] at (1.5864,0.8612,0.8612) {$\parallel$}; 
        \node[draw=none,rotate = 120] at (0.8165,0.8165,1.6330) {$\parallel$}; 
        \node[draw=none,rotate = 30] at (0.8165,1.6330,0.8165) {$\parallel$}; 

        \node[draw=none,rotate = 105] at (1.3765,0.4588,1.3765) {$\mid$}; 
        \node[draw=none,rotate = 0] at (0.4588,1.3765,1.3765) {$\mid$}; 
        \node[draw=none,rotate = 45] at (1.3765,1.3765,0.4588) {$\mid$}; 

        \tdplotdefinepoints(0,0,0)(\vax,\vay,\vaz)(\vbx,\vby,\vbz)
        \tdplotdrawpolytopearc[thick]{2}{}{} 
        
        \tdplotdefinepoints(0,0,0)(\vax,\vay,\vaz)(\vcx,\vcy,\vcz)
        \tdplotdrawpolytopearc[thick]{2}{}{}

        \tdplotdefinepoints(0,0,0)(\vax,\vay,\vaz)(\vcx,\vcy,\vcz)
        \tdplotdrawpolytopearc[thick]{0.5}{anchor = north west}{$\varphi$}

        \tdplotdefinepoints(0,0,0)(\vbx,\vby,\vbz)(\vcx,\vcy,\vcz)
        \tdplotdrawpolytopearc[thick]{2}{}{}

        \tdplotdefinepoints(0,0,0)(\vax,\vay,\vaz)(\vex,\vey,\vez)
        \tdplotdrawpolytopearc[thick]{2}{}{}

        \tdplotdefinepoints(0,0,0)(\vbx,\vby,\vbz)(\vex,\vey,\vez)
        \tdplotdrawpolytopearc[thick]{2}{}{}

        \tdplotdefinepoints(0,0,0)(\vcx,\vcy,\vcz)(\vex,\vey,\vez)
        \tdplotdrawpolytopearc[thick]{2}{}{}
    \end{tikzpicture}
     \caption{An equilateral triangle $\triangle ABC$ in $\P^2$ or equivalently an equiangular arrangement of three lines through the origin $\ell_A, \ell_B, \ell_C$ in $\mathbb{R}^3$. The angle of the arrangement is $\varphi$ and $E$ is the centroid.}
    \label{fig:equiangular-case}
\end{figure}
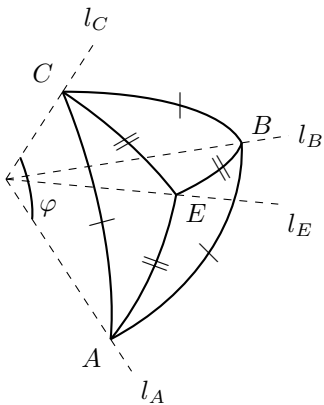

\begin{remark} \label{remark:Afsari}
If instead one considers in Problem \ref{problem:Fermat-Torricelli} the angular distance in $\P^2$, that is $\varphi_{P_1 P_2}$ for $P_1,P_2 \in \P^2$, then one can view $\P^2$ as a complete Riemannian manifold and the results of Afsari \cite{afsari2011riemannian} and Arnaudon \& Miclo \cite{arnaudon2014means} apply. In particular, it follows from Theorem 1 in \cite{afsari2011riemannian}, that the Fermat-Torricelli point of a projective triangle $\triangle ABC$ is unique, as soon as the lines $\ell_A, \, \ell_B, \, \ell_C$ pass through the interior of a spherical cap of the unit sphere of radius equal to $45^\circ$. Under the angular distance, uniqueness also follows from \cite{arnaudon2014means}, whenever $A,B,C$ are randomly chosen under the Lebesque measure on $\P^2 \times \P^2 \times \P^2$. 

For $P_1,P_2 \in \P^n$, \cite{chen2021search} used the distance $\sin\varphi_{P_1P_2}$ to study line packings of low-coherence in $\mathbb{R}^{n+1}$, known as Grassmannian frames. Packings in more general Grassmannians were studied in \cite{conway1996packing}, where it was observed that the sine distance produced more interesting packings than the angular distance. 

With regards to the Fermat-Torricelli problem in the projective plane, it is interesting to compare the effect of the distance choice on the location of the Fermat-Torricelli point. We leave such a study to future endeavors. 

\end{remark}

The paper is organized as follows. In \S \, \ref{section:first-properties} we reduce Problem \ref{problem:Fermat-Torricelli} to the unit sphere $\Sp^2 \subset \mathbb{R}^3$ and establish basic properties that will be used repeatedly in the sequel. In  \S \, \ref{section:big-triangles} we discuss a special type of projective triangles, that we call \emph{big triangles}, and prove that for such triangles the Fermat-Torricelli point is always a vertex. In  \S \, \ref{section:equiangular} we prove part (3) of Theorem \ref{thm:theorem}, the equilateral case. This is used in \S \, \ref{section:two-equal-angles} together with a deformation argument to prove part (2) of Theorem \ref{thm:theorem}, the isosceles case. Finally, \S \, \ref{section:general-case} uses these two cases to prove part (1) of Theorem \ref{thm:theorem}, the general case. 

The statement of the equilateral case was already contained in the expository paper \cite{tsakiris2017hyperplane}. A proof has only been published in the PhD thesis \cite{tsakiris2017dual}; the proof here is an improved version. We are grateful to Bijan Afsari for valuable discussions on geometric medians and thank Miklos P\'alfia for a useful discussion regarding this problem. Finally, we are grateful to the anonymous referee, whose scrutinous reading and constructive comments helped in improving the manuscript.


\section{Reduction to the Sphere and First Properties} \label{section:first-properties}

Let $A,B$ be two points in the real projective plane $\P^2$ and let $a,b$ be points on the unit sphere $\Sp^2$ of $\R^3$ that respectively represent $A,B$. Denote by $\ell_A$ the line in $\R^3$ that passes through the origin and $a$ and similarly for $\ell_B$. We denote by $\varphi_{ab}$ the angle between the vectors $a,b$ and by $\varphi_{AB} \le 90^\circ$ the smallest angle between $\ell_A$ and $\ell_B$. With $a^\top b$ the standard inner product of $a$ and $b$, we have that $\varphi_{AB}=\arccos(|a^\top b|)$ and $\varphi_{ab}=\arccos(a^\top b)$. Hence, the function $\dP:\P^2 \times \P^2 \rightarrow \R_{\ge 0}$ that appears in Problem \ref{problem:Fermat-Torricelli} induces a function $\dS:\Sp^2 \times \Sp^2 \rightarrow \R_{\ge 0}$ given by $$\dS(a,b) =\sqrt{1-(a^\top b)^2}= \sin \varphi_{AB}=\dP(A,B). $$ \noindent Note that $\dS(a,b)$ is the Euclidean distance of the point $a$ to the line $\ell_B$. 

Problem \ref{problem:Fermat-Torricelli} can be equivalently cast as a problem on the sphere
\begin{align}
    \min_{p \in \Sp^2} \, \, \, \J_{abc}(p) := \dS(a,p) + \dS(b,p) + \dS(c,p), \label{eq:FT-sphere}
\end{align} where now $c \in \Sp^2$ represents $C \in \P^2$. Following the assumption in Problem \ref{problem:Fermat-Torricelli}, the points $a,b,c$ are linearly independent. 

The function $\dS$ also satisfies the triangle inequality, and thus $d_{\mathbb{P}^2}$ is a genuine distance function (but $\dS$ is not, in view of $\dS(a,-a)=0$). 
 
\begin{lemma} \label{lem:triangle-inequality}
For any $v_1,v_2,v_3\in\Sp^2$ we have $\dS(v_1,v_2) \le \dS(v_1,v_3)+\dS(v_2,v_3)$, with the equality achieved if and only if $v_3=\pm v_1$ or $v_3=\pm v_2$.
\end{lemma}
\begin{proof}
If $v_1,v_2$ are collinear then the statement is trivial. So assume otherwise and let $\mathcal{H}$ be the plane $\vspan(v_1,v_2)$. Let $\varphi_{ij}=\arccos(v_i^\top v_j)$ for $i< j \in \{1,2,3\}$. Suppose first that $v_3\in \mathcal{H}$. The points $v_1,-v_1,v_2,-v_2$ partition the unit circle $\Sp^1 \subset \mathcal{H}$ into four arcs. The point $v_3$ lies in one of these arcs, say $\alpha$, whose endpoints we denote by $\bar{v}_1 \in \{v_1,-v_1\}$ and $\bar{v}_2 \in \{v_2,-v_2\}$. The length of $\alpha$ is either $\varphi_{12}$ or $\pi - \varphi_{12}$. The point $v_3$ partitions $\alpha$ into two arcs $\alpha_1$ and $\alpha_2$ with respective endpoints $\bar{v_1},v_3$ and $v_3, \bar{v_2}$. The length of $\alpha_1$ is either $\varphi_{13}$ or $\pi - \varphi_{13}$. Similarly, the length of $\alpha_2$ is either $\varphi_{23}$ or $\pi - \varphi_{23}$. Now,
\begin{align*}
    \dS(v_1,v_2) &= \sin (\varphi_{12}) \\
    &= \sin (\operatorname{length}(\alpha)) \\  
    &= \sin(\operatorname{length}(\alpha_1)+\operatorname{length}(\alpha_2)) \\
    &\le \sin(\operatorname{length}(\alpha_1)) + \sin(\operatorname{length}(\alpha_2)) \\
    & = \sin (\varphi_{13}) + \sin(\varphi_{23}) \\
    & = \dS(v_1,v_3) + \dS(v_2,v_3).
\end{align*}
The equality is achieved if and only if $\operatorname{length}(\alpha_1)=0$ or $\operatorname{length}(\alpha_2)=0$, which means $v_3=\bar{v}_1$ or $v_3=\bar{v}_2$. This proves the statement for $v_3 \in \mathcal{H}$.

To finish, consider any $v_3 \in \Sp^2 \setminus \mathcal{H}$ and let $\ell_{v_3}$ be the line in $\R^3$ spanned by $v_3$. Let $\pi_{\mathcal{H}}(\ell_{v_3})$ be the orthogonal projection of $\ell_{v_3}$ onto the plane $\mathcal{H}$. Let $v_3'$ be any of the intersection points of $\pi_{\mathcal{H}}(\ell_{v_3})$ with the unit circle $\Sp^1 \subset \mathcal{H}$. By elementary arguments one sees that $|v_1^\top v_3'| > |v_1^\top v_3|$ and $|v_2^\top v_3'| > |v_2^\top v_3|$. Hence 
$$\dS(v_1,v_3) + \dS(v_2,v_3) > \dS(v_1,v_3') + \dS(v_2,v_3').$$ By what we proved in the first paragraph, we have 
$$\dS(v_1,v_3') + \dS(v_2,v_3') \ge \dS(v_1,v_2),$$ which concludes the proof.
\end{proof}

Next, we establish some first properties of the set $\mathscr{P}_{abc}$ of solutions to problem \ref{eq:FT-sphere}, which we will need towards proving Theorem \ref{thm:theorem}.

\begin{lemma} \label{lem:distance-to-vertices}
Let $p\in \mathscr{P}_{abc}$ and $v_1,v_2$ two distinct points in $\{a,b,c\}$. Then $$\dS(v_1,p)\le \dS(v_1,v_2).$$
\end{lemma}
\begin{proof}
Write $\{v_3\} = \{a,b,c\} \setminus \{v_1,v_2\}$. By Lemma \ref{lem:triangle-inequality} we have $$\dS(v_2,p)+\dS(v_3,p)\ge \dS(v_2,v_3).$$
Since $p\in\mathscr{P}_{abc}$, we have
$$\dS(v_1,p)+\dS(v_2,p)+\dS(v_3,p) \le \dS(v_1,v_2)+\dS(v_2,v_3),$$
and necessarily $\dS(v_1,p)\le \dS(v_1,v_2)$.
\end{proof}

As in the proof of Lemma \ref{lem:triangle-inequality}, let us continue using $\mathcal{H}$ for the plane $\vspan(a,b)$ and $\pi_{\mathcal{H}}$ for the orthogonal projection of $\R^3$ onto $\mathcal{H}$.

\begin{lemma} \label{lem:projections}
Let $p\in \mathscr{P}_{abc}$ and set $p_{\mathcal{H}} = \pi_{\mathcal{H}}(p)$ and $p_{\mathcal{H}^\perp}=p-p_{\mathcal{H}}$. Then $$(c^\top p_{\mathcal{H}}) (c^\top p_{\mathcal{H}^\perp}) \ge 0.$$
\end{lemma}
\begin{proof}
For the sake of a contradiction suppose that $(c^\top p_{\mathcal{H}}) (c^\top p_{\mathcal{H}^\perp}) < 0$. Let $p'=p_{\mathcal{H}}-p_{\mathcal{H}^\perp}\in\Sp^2$ be the reflection of $p$ with respect to $\mathcal{H}$. Then $a^\top p'=a^\top p$, $b^\top p'=b^\top p$ and 
$$|c^\top p'|=|c^\top p_{\mathcal{H}}-c^\top p_{\mathcal{H}^\perp}|>|c^\top p_{\mathcal{H}}+c^\top p_{\mathcal{H}^\perp}|=|c^\top p|.$$ Therefore, $\dS(a,p')=\dS(a,p)$, $\dS(b,p')=\dS(b,p)$ and $\dS(c,p')<\dS(c,p)$, that is $\J_{abc}(p')<\J_{abc}(p)$. This contradicts the fact that $p\in \mathscr{P}_{abc}$. 
\end{proof}

\begin{lemma} \label{lem:inner-products-of-normal}
Let $p\in \mathscr{P}_{abc}$ and let $\{v_1,v_2,v_3\}=\{a,b,c\}$. Let $n \in \R^3$ be a non-zero normal vector to the plane $\vspan(v_1,v_2)$. Then $v_3^\top p$ and $(v_3^\top n)(p^\top n)$ are both either non-negative or non-positive.
\end{lemma}
\begin{proof}
It is enough to prove the statement for $v_1=a, \, v_2=b, \, v_3=c$. 
By Lemma \ref{lem:projections}, $(c^\top p_{\mathcal{H}}) (c^\top p_{\mathcal{H}^\perp})\ge 0$. If $c^\top p = c^\top p_{\mathcal{H}} + c^\top p_{\mathcal{H}^\perp} \ge 0$, we have $c^\top p_{\mathcal{H}}\ge 0$ and $c^\top p_{\mathcal{H}^\perp} \ge 0$. Moreover $p_{\mathcal{H}^\perp}=(p^\top \hat{n}) \hat{n}$, where $\hat{n} = n/\|n\|_2$. Taking inner product of both sides of this equality with $c$, we get $c^\top p_{\mathcal{H}^\perp}  = (c^\top \hat{n})(p^\top \hat{n})$. Thus $(c^\top n)(p^\top n)\ge 0$. The case $c^\top p \le 0$ follows similarly. 
\end{proof}

\begin{lemma} \label{lem:non-positive-non-negative}
Let $p\in \mathscr{P}_{abc}$ and suppose that $a^\top b,\, a^\top c,\, b^\top c$ are all positive. Then $a^\top p, \, b^\top p, \, c^\top p$ are either all non-positive or all non-negative.
\end{lemma}
\begin{proof}

At least two of $a^\top p, \, b^\top p, \, c^\top p$ are both non-negative or non-positive. Without loss of generality we assume that $a^\top p$ and $b^\top p$ are non-negative. We may assume that not both of them are zero, since otherwise we are done. We will prove that $c^\top p \ge 0$.

In the notation of Lemma \ref{lem:projections} we have $p_{\mathcal{H}}=\lambda_a a+\lambda_b b$ for some $\lambda_a, \lambda_b \in \R$. We will prove that $\lambda_a, \lambda_b \ge 0$. By Lemma \ref{lem:distance-to-vertices} we have 
\begin{align*}
&1 \ge a^\top p_{\mathcal{H}} = \lambda_a + \lambda_b \, a^\top b =a^\top p\ge a^\top b \\
&1 \ge b^\top p_{\mathcal{H}} = \lambda_a \, a^\top b + \lambda_b =b^\top p\ge a^\top b,
\end{align*} from which we extract
\begin{align*}
& 1\ge \lambda_a+\lambda_b \, a^\top b \ge a^\top b\\
&a^\top b \le \lambda_a \, a^\top b + \lambda_b \le 1.
\end{align*}
Subtracting the first row of inequalities from the second row we get
$$ a^\top b-1\le(1-a^\top b)(\lambda_b-\lambda_a)\le 1-a^\top b, $$
which in turn gives $-1\le\lambda_b-\lambda_a\le 1$. Combining the inequalities $a^\top b \le \lambda_a + \lambda_b \, a^\top b$ and $\lambda_b \le 1 + \lambda_a$ we have 
\begin{align} 
a^\top b \le \lambda_a + \lambda_b \, a^\top b \le \lambda_a +(1+\lambda_a) a^\top b, \label{eq:lambda-a-b}
\end{align} \noindent from which we read $(1+a^\top b) \lambda_a \ge 0$ and thus $\lambda_a \ge 0$. Similarly, we have $\lambda_b \ge 0$. 

As a consequence,
$$c^\top p_{\mathcal{H}}=\lambda_a \, a^\top c+\lambda_b \, b^\top c\ge 0.$$
Since by hypothesis $a^\top b>0$, \eqref{eq:lambda-a-b} gives that not both $\lambda_a$ and $\lambda_b$ are zero. Thus the hypothesis $a^\top c, \, b^\top c>0$ gives $c^\top p_{\mathcal{H}} > 0$. Now Lemma \ref{lem:projections} gives $ c^\top p_{\mathcal{H}^\perp} \ge 0$, whence $c^\top p = c^\top p_{\mathcal{H}}+c^\top p_{\mathcal{H}^\perp} > 0$.
\end{proof}

Next, we develop a first order optimality condition. For convenience, with $p \in \Sp^2$ we define a function $\tau_p: \Sp^2 \setminus \{\pm p\} \rightarrow \R^3$ as 
$$\tau_p(x)=\frac{(I-pp^\top)x}{\|(I-pp^\top)x\|_2}=\frac{(I-pp^\top)x}{\sqrt{1-(p^\top x)^2}},$$ where $I$ is the $3 \times 3$ identity matrix. The condition is:

\begin{lemma} \label{lem:zero-gradient}
    If $p\in\mathscr{P}_{abc}$ and $p\notin\{\pm a,\pm b,\pm c\}$, then
    $$(a^\top p)\tau_p(a)+(b^\top p)\tau_p(b)+(c^\top p)\tau_p(c)=0$$
\end{lemma}

\begin{proof}
    The formula for $\J_{abc}:\Sp^2 \rightarrow \R$ given in \eqref{eq:FT-sphere} can also be used to define a function $\J_{abc}':\mathcal{C} \rightarrow \R$, where $\mathcal{C}$ is the convex polytope
    $$\mathcal{C}= \left\{\xi \in\R^3 \, \mid \, |a^\top \xi|\le 1, \, |b^\top \xi|\le 1,\, |c^\top \xi|\le 1\right\}.$$
    This function is smooth everywhere except at the boundary of the polytope    
    $$\partial \mathcal{C} = \left\{\xi \in \mathcal{C} \, : \,   |a^\top \xi |= 1\text{ or }|b^\top \xi |= 1\text{ or }|c^\top \xi |= 1\right\}.$$                                                                                                   
    Its gradient at a point $q \in \mathcal{C} \setminus \partial \mathcal{C}$ is given by 
    \begin{align*}
    \nabla_q\J_{abc}'&=
        -\frac{\left(a^\top q\right)a}{\sqrt{1-\left(a^\top q\right)^2}}
        -\frac{\left(b^\top q\right)b}{\sqrt{1-\left(b^\top q\right)^2}}
        -\frac{\left(c^\top q\right)c}{\sqrt{1-\left(c^\top q\right)^2}}.        
    \end{align*} Now, $\Sp^2 \subset \mathcal{C}$ is a smooth Riemannian manifold and the punctured sphere $\Sp^{2}_{abc}:=\Sp^2 \setminus \partial \mathcal{C} \cap \Sp^2 = \Sp^2 \setminus \{\pm a,\pm b,\pm c \} $ is a smooth open submanifold. Hence $\J_{abc}$ is smooth on $\Sp^{2}_{abc}$ with Riemann gradient at a point $q \in \Sp^{2}_{abc}$ given by $$\nabla_q\J_{abc} = (I-qq^\top) \nabla_q\J_{abc}'.$$ Let us call $\J_{abc}''$ the restriction of $\J_{abc}$ to $\Sp^{2}_{abc}$. Since $p \in \mathscr{P}_{abc} \setminus \{\pm a,\pm b,\pm c \}$, it is a global minimum of the smooth function $\J_{abc}''$ and thus basic facts from Riemannian optimization imply that $p$ must satisfy the equation $$ \nabla_p\J_{abc}''= \nabla_p\J_{abc} = (I-pp^\top) \nabla_p\J_{abc}' = 0,$$ which is exactly the relationship in the statement of the lemma.
\end{proof}

\section{Big Triangles} \label{section:big-triangles}

As in the previous section, let $a, b, c \in \Sp^2$ be representatives of $A, B, C \in \P^2$ and observe that the sign of the quantity $(a^\top b)(b^\top c)(a^\top c)$ is independent of the choice of representatives. The following notion will turn out to be convenient:

\begin{definition} \label{dfn:big-triangle}
We call the projective triangle $\triangle ABC$ a big triangle, if $$(a^\top b)(b^\top c)(a^\top c)\le 0.$$
\end{definition}

If the inequality in Definition \ref{dfn:big-triangle} is strict, the corresponding line arrangement was called \emph{coherent} in \cite{taylor1977regular}. Instead, the attribute \emph{big} here is justified by: 

\begin{proposition} \label{prp:"big"->big}
    If $\triangle ABC$ is a big triangle, then $\varphi_{AB}+\varphi_{AC}+\varphi_{BC} > \pi$. 
\end{proposition}
\begin{proof}
Recall our convention $\varphi_{AB}\le \varphi_{AC}\le \varphi_{BC}$. If $\varphi_{AC}=\pi/2$, then $\varphi_{BC}=\pi/2$ and we are done. We may thus assume $\varphi_{AB}\le\varphi_{AC}<\pi/2$. We may also assume that $a^\top b = \cos\varphi_{AB}$ and $a^\top c= \cos\varphi_{AC}$. Hence $a^\top b, \, a^\top c>0$ and so necessarily $b^\top c\le 0$. With $v_1, v_2 \in \Sp^2$ we denote by $\varphi_{v_1v_2}$ the angle between the vectors $v_1,v_2$. We have $\varphi_{bc}=\pi-\varphi_{BC}$ and the triangular inequality for arc lengths
$$\varphi_{ab}+\varphi_{ac}\ge \varphi_{bc},$$ gives the inequality in the statement with $\ge 0$. 
Note that the equality is achieved if and only if $a \in \Arc{bc}$, which would imply that $a,b,c$ are coplanar. 
\end{proof} 

The main result of this section is: 

\begin{proposition} \label{prp:big-triangle}
    If $\triangle ABC$ is a big triangle, then $\mathscr{P}_{ABC}\subset\{A,B,C\}$.
\end{proposition}

Towards proving this fact, we let $p$ be a representative of $P \in\mathscr{P}_{ABC}$ and assume without loss of generality that $a,b,c$ are such that $a^\top p, \,  b^\top p, \, c^\top p\ge 0$. Again without loss of generality we assume 
\begin{align}
a^\top p\ge \max\{b^\top p,c^\top p\} \, \, \, \Leftrightarrow \, \, \, \dS(a,p)\le \min \{ \dS(b,p),\dS(c,p) \}. \label{eq:big-triangle-hypothesis}
\end{align} 

\begin{lemma} \label{lem:a-top-p}
We have $a^\top p > 1/\sqrt{2}$; in particular $\varphi_{ap}=\varphi_{AP} <\pi/4$.
\end{lemma}
\begin{proof}
    If $(a^\top p)^2\le 1/2$, then $\dS(a,p)\ge \sqrt{2}/{2}$. By the assumption 
    $$\J_{abc}(p)\ge 3\dS(a,p) \ge \frac{3}{2}\sqrt{2}>2\ge \sin\varphi_{AB}+\sin\varphi_{AC}=\J_{abc}(a),$$
    which contradicts the fact that $p\in\mathscr{P}_{abc}$.
\end{proof}

\begin{lemma} \label{lem:nonpositive-b-top-c}
    There is a choice of representatives additionally satisfying $b^\top c\le 0$.
\end{lemma}
\begin{proof}
Recall the assumption $a^\top p\ge \max\{b^\top p,c^\top p\}$ and suppose first that equality is achieved. So suppose that 
$a^\top p=b^\top p$; the case $a^\top p=c^\top p$ follows similarly. From Lemma \ref{lem:a-top-p}, we know that $\varphi_{ap}=\varphi_{bp}<\pi/4$, and so $\varphi_{ab} \le \varphi_{ap}+\varphi_{bp} < \pi/2$. Hence, $\varphi_{ab}=\varphi_{AB}$ and $a^\top b > 0$. Since the triangle is big, either $a^\top c\le 0$ or $b^\top c\le 0$, and we can choose $c$ such that $b^\top c\le 0$. 

We may thus assume that $a^\top p > \max \{ b^\top p,c^\top p\}$. If $b^\top c \le 0$ we are done, so assume $b^\top c>0$. Since the triangle is big, we have either $a^\top b\ge 0, \, a^\top c\le 0$ or $a^\top b\le 0, \, a^\top c\ge 0$. We consider only the first case; the second case is similar.
    
    By the assumption, $a^\top p > b^\top p \ge 0$. Let $d=a-b$, then $d^\top p > 0$ and $d^\top c = a^\top c-b^\top c < 0$. Let $\V_d=\vspan(d)$ and $\V_{d^\perp}$ the orthogonal complement of $\V_d$. Then we have the decomposition $p=p_d+p_{d^\perp}$, where $p_d\in\V_d$ and $p_{d^\perp}\in\V_{d^\perp}$. Since $(a-b)^\top(a+b)=0$, we have
    \begin{align*}
        a^\top p_d &= (a+b-b)^\top p_d = - b^\top p_d, \\
        a^\top p_{d^\perp} &= (a-b+b)^\top p_{d^\perp} = b^\top p_{d^\perp}.
    \end{align*}
    From $p_d=(dd^\top)p$ we obtain $c^\top p_d =(d^\top c)(d^\top p) < 0$. Now $c^\top p=c^\top p_d+c^\top p_{d^\perp}$ and $c^\top p \ge 0$ by hypothesis, so $c^\top p_{d^\perp}\ge 0$. Let $p'=p_{d^\perp}-p_d$. Then 
    \begin{align*}
        a^\top p' &= a^\top p_{d^\perp} - a^\top p_d = b^\top p_{d^\perp} + b^\top p_d = b^\top p, \\
        b^\top p' &= b^\top p_{d^\perp} - b^\top p_d = a^\top p_{d^\perp} + a^\top p_d = a^\top p, \\
        c^\top p' &= c^\top p_{d^\perp} - c^\top p_d > c^\top p_{d^\perp} + c^\top p_d = c^\top p.
    \end{align*}
    But this implies the contradiction $\J_{abc}(p')<\J_{abc}(p)$.
\end{proof}

We can now finish the proof of Proposition \ref{prp:big-triangle}: 

By Lemma \ref{lem:nonpositive-b-top-c}, we have $b^\top c\le 0$.
    From the assumption $\varphi_{AP}\le \min(\varphi_{BP},\varphi_{CP})$, we have $p\ne\pm b$ and $p\ne\pm c$. Since $a^\top p\ge 0$, we also have $p\ne -a$. We assume that $p\ne a$ and reach a contradiction. Lemma \ref{lem:zero-gradient} gives $$(a^\top p)\tau_p(a)+(b^\top p)\tau_p(b)+(c^\top p)\tau_p(c)=0.$$ With $u=(b^\top p)\tau_p(b)+(c^\top p)\tau_p(c)$ the above equation can be written as $(a^\top p)\tau_p(a) + u = 0$. In what follows we show that $\|u\|_2<a^\top p$; this is a contradiction because by definition $\|\tau_p(a)\|_2 = 1$. From the definition of $u$, we have
    $$u^\top u = (b^\top p)^2 + (c^\top p)^2 + 2(b^\top p)(c^\top p)\tau_p(b)^\top \tau_p(c).$$
Using $b^\top c \le 0$ and the definition of the function $\tau_p$, one easily verifies that $$\tau_p(b)^\top \tau_p(c) \le -(b^\top p)(p^\top c).$$ In turn, this gives
\begin{align}
u^\top u \le (b^\top p)^2 + (c^\top p)^2 - 2(b^\top p)^2(c^\top p)^2 = (b^\top p)^2 + (c^\top p)^2 [1 - 2(b^\top p)^2]. \label{eq:big-triangle-uu^T}
\end{align}
From Lemma \ref{lem:a-top-p}, $(a^\top p)^2>1/2$. If both $(b^\top p)^2 \le 1/2$ and $(c^\top p)^2 \le 1/2$, then
\begin{align*}
    u^\top u \le (b^\top p)^2 + \frac{1}{2}(1 - 2(b^\top p)^2) = \frac{1}{2} < (a^\top p)^2,
\end{align*} and we are done. Suppose then that either $(b^\top p)^2 > 1/2$ or $(c^\top p)^2 > 1/2$. If $(b^\top p)^2 > 1/2$, then
\begin{align*}
    u^\top u \le  (b^\top p)^2 \le (a^\top p)^2,
\end{align*} where the last inequality is due to the hypothesis \eqref{eq:big-triangle-hypothesis}. If equality is achieved everywhere, \eqref{eq:big-triangle-uu^T} gives $c^\top p=0$, and we also have $b^\top p=a^\top p$. Now the triangular inequality (Lemma \ref{lem:triangle-inequality}) gives the contradiction $$\J_{abc}(p)=\dS(a,p)+\dS(b,p)+1>\sin\varphi_{AB}+1 \ge \J_{abc}(a).$$


\section{Equilateral Triangles} \label{section:equiangular}

In this section we will prove part (3) of Theorem \ref{thm:theorem}, which we will use in later sections to prove parts (1) and (2). This is the equilateral case where $\varphi_{AB}=\varphi_{AC}=\varphi_{BC} = \varphi$. If $\triangle ABC$ is a big triangle in the sense of Definition \ref{dfn:big-triangle}, then Proposition \ref{prp:"big"->big} gives $\varphi >60^\circ$. Then Proposition \ref{prp:big-triangle} together with the fact that $\J_{ABC}(A) = \J_{ABC}(B) = \J_{ABC}(C)$ shows that $\mathscr{P}_{ABC} = \{A,B,C\}$. We may thus assume in the rest of this section that the triangle $\triangle ABC$ is not big, that is $$(a^\top b)(a^\top c)(b^\top c)>0,$$ for any representatives $a,b,c$. We may also assume that $b,c$ are conveniently chosen such that $a^\top b =a^\top c =\cos\varphi$. From the above inequality $b^\top c>0$ and so $$a^\top b = a^\top c=b^\top c=\cos\varphi.$$

A first consequence of the symmetry of the configuration is:

\begin{lemma} \label{lem:y1-y2-y3}
Let $p\in \mathscr{P}_{abc}$ and set 
\begin{align*}
x_1=a^\top p, \, \, \,  x_2=b^\top p, \, \, \, x_3=c^\top p, \, \, \,\text{and} \, \, \,   y_i = \sqrt{1-x_i^2}, \, \, \, \text{for} \, \, \,  i=1,2,3.
\end{align*}
Then either $y_1=y_2$ or $y_1=y_3$ or $y_2=y_3$.
\end{lemma}

\begin{proof}
    If $p$ is one of $\pm a,\pm b,\pm c$, then the statement clearly holds, since if say $p=a$, then $y_2 = y_3 = \sin(\varphi)$. So suppose that $p \notin \{\pm a,\pm b,\pm c\}$. Lemma \ref{lem:zero-gradient} gives $$(a^\top p)\tau_p(a)+(b^\top p)\tau_p(b)+(c^\top p)\tau_p(c)=0.$$   
    Setting $z=\cos\varphi$, and taking inner product of both sides of the above equation with $a,b,c$, we respectively obtain
    \begin{align*}
        &\frac{x_1}{y_1} +\frac{x_2}{y_2}z +\frac{x_3}{y_3}z - \bigg(\frac{x_1}{y_1}x_1 +\frac{x_2}{y_2}x_2 +\frac{x_3}{y_3}x_3\bigg)x_1 = 0\\
        &\frac{x_1}{y_1}z +\frac{x_2}{y_2} +\frac{x_3}{y_3}z - \bigg(\frac{x_1}{y_1}x_1 +\frac{x_2}{y_2}x_2 +\frac{x_3}{y_3}x_3\bigg)x_2 = 0\\
        & \frac{x_1}{y_1}z +\frac{x_2}{y_2}z +\frac{x_3}{y_3} - \bigg(\frac{x_1}{y_1}x_1 +\frac{x_2}{y_2}x_2 +\frac{x_3}{y_3}x_3\bigg)x_3 = 0.
    \end{align*} Multiplying the first of the above equations with $y_1 y_2 y_3$ and using the fact that $1-x_1^2 = y_1^2$, gives
$$ x_1 y_1^2 y_2 y_3 + x_2 y_1 y_3(z-x_1x_2)+x_3y_1y_2(z-x_1x_3) = 0. $$ Since by hypothesis $\varphi <90^\circ$, we can cancel $y_1$ to get 
$$ x_1 y_1 y_2 y_3 + x_2 y_3(z-x_1x_2)+x_3y_2(z-x_1x_3) = 0. $$            
We can extract similar polynomial relations from the other two equations. Formally, let $\mathcal{T} = \R[X_1,X_2,X_3,Y_1,Y_2,Y_3,Z]$ be a polynomial ring in seven variables over the real numbers. Then for $i=1,2,3$ the $x_i,y_i$ and $z$ are roots of the following polynomials of $\mathcal{T}$:
        \begin{align*}
    F_1 &=X_1Y_1Y_2Y_3+X_2Y_3(Z-X_1X_2)+X_3Y_2(Z-X_1X_3) \\
    F_2 &=X_1Y_3(Z-X_1X_2)+X_2Y_1Y_2Y_3+X_3Y_1(Z-X_2X_3) \\
    F_3 &=X_1Y_2(Z-X_1X_3)+X_2Y_1(Z-X_2X_3)+X_3Y_1Y_2Y_3 \\
    F_4 &= X_1^2+Y_1^2-1 \\
    F_5 &= X_2^2+Y_2^2-1 \\
    F_6 &= X_3^2+Y_3^2-1
    \end{align*} 

Let $\mathcal{I}$ be the ideal of $\mathcal{T}$ generated by the $F_i, \, i=1,\dots,6$. Define the polynomial $$F=(1-Z)(Y_1-Y_2)(Y_2-Y_3)(Y_3-Y_1)(Y_1+Y_2+Y_3) \in \mathcal{T}.$$ We claim that $F \in \mathcal{I}$, that is, there exist $C_1,\dots,C_6 \in \mathcal{T}$ such that 
\begin{align}    
F = C_1 F_1 + C_2 F_2 + C_3 F_3 + C_4 F_4 + C_5 F_5 + C_6 F_6. \label{eq:F-F_i}
\end{align} Since the $x_i,y_i, z$ are roots of the $F_1,\dots,F_6$, we see from \eqref{eq:F-F_i} that they are roots  of $F$. Hence, $$(1-z)(y_1-y_2)(y_2-y_3)(y_3-y_1)(y_1+y_2+y_3)=0.$$ Now, $\psi >0 $ and so $z < 1$. Also, by definition $y_1,y_2,y_3$ are non-negative and by the linear independence assumption on $a,b,c$ not all of them can be zero. That is, 
    \begin{align*}
        (y_1-y_2)(y_1-y_3)(y_2-y_3) = 0,
    \end{align*} which is exactly the statement of the lemma. 

In what follows we prove the above claim. First, note that 
\begin{align}
F &= Y_1^3Y_2(Z-1)+Y_1Y_2^3(1-Z)+Y_1^3Y_3(1-Z) \nonumber \\
&+Y_2^3Y_3(Z-1)+Y_1Y_3^3(Z-1)+Y_2Y_3^3(1-Z). \nonumber
\end{align}  
It is enough to show that 
\begin{align}
F'&=F_4 [Y_2^3Y_3(Z-1) + Y_2Y_3^3(1-Z)] + F_5[Y_1^3Y_3(1-Z) + Y_1Y_3^3(Z-1)] \nonumber \\  &+F_6[Y_1^3Y_2(Z-1) + Y_1Y_2^3(1-Z)] + F \in \mathcal{I}. \nonumber
\end{align} Replacing $Y_i^3$ by $Y_i(1-X_i^2)$ for each $i=1,2,3$ in $F'$ does not change the class of $F'$ modulo $\mathcal{I}$, hence it is enough to show that
\begin{align}
F'' = &(X_1^2+Y_1^2)Y_2Y_3(Z-1) (X_3^2-X_2^2) + (X_2^2+Y_2^2)Y_1Y_3(Z-1) (X_1^2-X_3^2) + \nonumber \\ &(X_3^2+Y_3^2)Y_1Y_2(Z-1) (X_2^2-X_1^2) \in \mathcal{I}. \nonumber
\end{align} We also write $F'' = (Z-1) F'''$. Now, we observe that 
\begin{align}
&F''' + X_1(Y_3-Y_2)F_1+X_2(Y_1-Y_3)F_2+X_3(Y_2-Y_1)F_3 = \nonumber \\
&Z[X_1X_2Y_3(Y_1-Y_2)+X_1X_3Y_2(Y_3-Y_1)+X_2X_3Y_1(Y_2-Y_3)]\nonumber, 
\end{align} where we denote the factor in the parenthesis on the right-hand-side by $F''''$. Hence, the class of $F'$ modulo $\mathcal{I}$ is the same as the class of $Z(Z-1)F''''$. In fact, already $(Z-1) F'''' \in \mathcal{I}$. To see this, we observe that 
\begin{align}
(Z-1)F'''' = &Y_1(X_2-X_3)F_1+Y_2(X_3-X_1)F_2+Y_3(X_1-X_2)F_3 +\nonumber \\
&X_1Y_2Y_3(X_3-X_2)F_4 + X_2Y_1Y_3(X_1-X_3)F_5+ X_3Y_1Y_2(X_2-X_1)F_6. \nonumber
\end{align} This concludes the proof of Lemma \ref{lem:y1-y2-y3}.
\end{proof}

We are now ready to prove part (3) of Theorem \ref{thm:theorem}:

\begin{proposition}
Set $e = (a+b+c)/\|a+b+c\|_2 \in \Sp^2$ and let $E$ be the point of $\P^2$ represented by $e$. Then $\mathscr{P}_{ABC}$ satisfies the following phase transition:
    \begin{enumerate}
        \item If $\varphi > 60^\circ$, then $\mathscr{P}_{ABC} = \{A,B,C\}$.
        \item If $\varphi = 60^\circ$, then $\mathscr{P}_{ABC} = \{A,B,C, E\}$.
        \item If $\varphi < 60^\circ$, then $\mathscr{P}_{ABC} = \{E\}$.
    \end{enumerate}
    \label{prop:equiangular}
\end{proposition}
\begin{proof}
Recall that by definition $0 < \varphi < 90^\circ$. We can parametrize $a,b,c$ as
\begin{align*}
    a &= \mu\mat{\alpha&1&1}^\top, \\
    b &= \mu\mat{1&\alpha&1}^\top, \\
    c &= \mu\mat{1&1&\alpha}^\top,
\end{align*}
with $\mu := \left(2+\alpha^2\right)^{-\frac{1}{2}}$. The parameter $\alpha$ encodes the value of the angle $\varphi$ as follows:	 
\begin{align*}
    \cos(\varphi)=\frac{1+2\alpha}{2+\alpha^2}, \, \, \, \, \, \,   \alpha = \frac{1}{\cos\varphi} \pm \sqrt{\frac{1}{\cos^{2}\varphi}+\frac{1}{\cos\varphi}-2}
\end{align*}
From the second formula above, selecting the largest of the two values for $\alpha$ gives a $1-1$ correspondence between $\varphi \in (0^\circ, 90^\circ)$ and $\alpha \in (1,+\infty)$. With this convention in the sequel $\alpha$ will range in the interval $(1,+\infty)$ with
\begin{itemize}
\item $\alpha \in (4,+\infty)$ if and only if $\varphi \in (60^\circ,90^\circ)$, 
\item $\alpha = 4$ if and only if $\varphi = 60^\circ$, and
\item $\alpha \in (1,4)$ if and only if $\varphi \in (0^\circ, 60^\circ)$.
\end{itemize}	 

By the above parametrization and the hypothesis that $\varphi<90^\circ$, we see that $a^\top b, \, a^\top c, \, b^\top c$ are all positive. Thus by Lemma \ref{lem:non-positive-non-negative} we have that $a^\top p, \, b^\top p, \, c^\top p$ are either all non-negative or all non-positive. Hence, by Lemma \ref{lem:y1-y2-y3} at least two among $a^\top p, \, b^\top p, \, c^\top p$ are equal. Which two, depends on $P$. Let us assume for a moment that $P$ is such that $a^\top p=b^\top p $. This equation places $p$ inside the plane
\begin{align*}
\mathcal{V}_{ab} = \vspan\left(\mat{1\\1\\0},\mat{0\\0\\1}\right).
\end{align*} The plane $\mathcal{V}_{ab}$ consists of all vectors that have equal angles from $a$ and $b$; in particular both $c$ and $e$ are in $\mathcal{V}_{ab}$. In the rest of the proof we will show that, 
if $\alpha>4$ then $P=C$; if $\alpha=4$ then $P\in\{C,E\}$; if $\alpha<4$ then $P=E$.
The cases $p \in \mathcal{V}_{ac}$ or $p \in \mathcal{V}_{bc}$ are treated in an identical manner and this concludes the proof because $\mathcal{J}_{ABC}(A)=\mathcal{J}_{ABC}(B)=\mathcal{J}_{ABC}(C)$.

Since $p \in \mathcal{V}_{ab} \cap \Sp^2$, we have the parametrization
\begin{align*}
p = \frac{1}{\sqrt{2v^2+w^2}} \begin{bmatrix} v \\ v \\w\end{bmatrix},
\end{align*} with $v, \, w \in \R$. The choice $v=0$, corresponding to $p=e_3$ (the third standard basis vector) can be excluded, since moving $p$ from $e_3 \in \mathcal{V}_{ab} \cap \Sp^2$ to $c \in \mathcal{V}_{ab} \cap \Sp^2$ while staying in $\mathcal{V}_{ab} \cap \Sp^2$, results in decreasing angles of $p$ to $a,b,c$ and thus in lower values of $\mathcal{J}_{abc}(p)$. Consequently, we can assume $v=1$, and write
\begin{align*}
p = \frac{1}{\sqrt{2+w^2}} \begin{bmatrix} 1 \\ 1 \\w\end{bmatrix},
\end{align*} with $w \in \R$. Problem \ref{eq:FT-sphere} is now reduced to minimizing the function
\begin{align}
	\J(w) = \frac{2\left[(2+w^2)(2+\alpha^2)-(1+\alpha+w)^2\right]^{1/2}+\sqrt{2}\left|\alpha-w\right|}
	{\left[(2+w^2)(2+\alpha^2) \right]^{1/2}}. \label{eq:Jw}
\end{align}
It is easy to check that:
\begin{itemize}
    \item For $\alpha>1$, the following quantity is always positive: 
    \begin{align}
        u:=(2+w^2)(2+\alpha^2)-(1+\alpha+w)^2 = (\alpha-1)^2+(w-1)^2+(\alpha w-1)^2 \label{eq:u}
    \end{align}
    \item The choice $w = \alpha$ corresponds to $p=c$, and that is the only point where $\J(w)$ is non-differentiable.
    \item The choice $w=1$ corresponds to $p=e = (a+b+c)/\|a+b+c\|_2$.		
    \item $\J(c) = \J(e)=\sqrt{3}$ when $\alpha = 4$.
\end{itemize}
	
We consider the three intervals of $\alpha$ mentioned above corresponding to statements (1), (2) and (3) of the proposition.

First we consider $\alpha \in (4,+\infty)$. We will show that for $w \neq \alpha$, it is always the case that $\J(w)>\J(\alpha)$ so that $P=C$.   
Expanding the inequality $\J(w)>\J(\alpha)$, we obtain
\begin{align}
\frac{2u^{1/2}+\sqrt{2}\left|\alpha-w\right|}
{\left[(2+w^2)(2+\alpha^2) \right]^{1/2}}
>\frac{2\sqrt{(2+\alpha^2)^2-(1+2\alpha)^2}}{2+\alpha^2}. \label{eq:J(w)>J(alpha)}
\end{align}
Squaring \eqref{eq:J(w)>J(alpha)} we have that $\J(w)>\J(\alpha)$ if and only if 
\begin{align} \label{eq:p1Inequality}
    4\sqrt{2} u^{1/2}\left|\alpha-w\right|(2+\alpha^2)> p_1,
\end{align}
where 
$$p_1 := 4(2+w^2)\left[(2+\alpha^2)^2-(1+2\alpha)^2\right]-(2+\alpha^2)\left[4u+2(\alpha - w)^2\right].$$
By squaring \eqref{eq:p1Inequality} we have that $\J(w)>\J(\alpha)$ if
\begin{align*}
    p_2:= 32u(\alpha-w)^2(2+\alpha^2)^2-p_1^2 >0.
\end{align*}
Now $p_2$ admits the factorization
\begin{align*}
    p_2=&\ 4(w - \alpha)^2 p_3, \\
    p_3\coloneqq&\ (8\alpha^6 - 9\alpha^4 - 112\alpha^3 + 32) w^2 + (- 30\alpha^5 - 88\alpha^4 + 40\alpha^3 + 240\alpha^2 + 64\alpha - 64)w \\
    &+(7\alpha^6 - 24\alpha^5 + 64\alpha^4 + 48\alpha^3 + 48\alpha^2 - 256\alpha + 32),     
\end{align*}
where $p_3$ as a quadratic in $w$ has discriminant
\begin{align*}
\Delta(p_3)&= -32\alpha(\alpha-4)(\alpha-1)^2(\alpha^2+2)^{2}(\alpha^2+2\alpha+3)(7\alpha^2+4\alpha+16).
\end{align*}
One checks directly that $\Delta(p_3)<0$ for $\alpha\in(4,+\infty)$. Moreover, for $\alpha\in(4,+\infty)$ we have
\begin{align*}
    8\alpha^6 -9\alpha^4 - 112 \alpha^3 + 32 > 128\alpha^4 - 9\alpha^4 - 112\alpha^4 + 32 > 0,
\end{align*}
Hence for $\alpha\in(4,+\infty)$ the leading coefficient $8\alpha^6 - 9\alpha^4 - 112\alpha^3 + 32$ of $p_3$ is always positive. Therefore, for $\alpha\in(4,+\infty)$ we have that $p_3$ is always positive, thus $p_2$ is positive when $w\ne \alpha$. That is $\mathcal{J}(w) > \mathcal{J}(\alpha)$ for any $w \neq \alpha$, that is $P = C$.  

Moving on to the case $\alpha = 4$, we have
\begin{align*}
    p_2|_{\alpha=4} = 93312(w-4)^2(w-1)^2. 
\end{align*}
Since $p_2 >0$ implies that $\J(w)>\J(4)$, we see that the minimum of $\J(w)$ can only be attained for $w \in \{1,4\}$. Since $\J(1)=\J(4)=\sqrt{3}$, we conclude that $w=1,4$ are the only minima of $\J(w)$.

For $\alpha\in(1,4)$ we proceed in a similar manner to show that $\J(w)>\J(1)$ for any $w \neq 1$. Expanding the inequality $\J(w)>\J(1)$ we have
\begin{align*}
\frac{2u^{1/2}+\sqrt{2}\left|\alpha-w\right|}
{\left[(2+w^2)(2+\alpha^2) \right]^{1/2}}
>\frac{3\sqrt{2}|\alpha-1|}{\sqrt{3(2+\alpha^2)}}.
\end{align*} Both sides are non-negative and so after squaring them we have that $\J(w)>\J(1)$ if and only if 
\begin{align}\label{eq:q1Inequality}
    q_1<4\sqrt{2} u^{1/2}\left|\alpha-w\right|, 
\end{align} where $q_1$ is defined as 
\begin{align*}
q_1 := 6(2+w^2)(\alpha-1)^2-\left[4u+2(\alpha - w)^2\right].
\end{align*}

If $q_1<0$, then we are done. So suppose in the rest of the proof that $q_1 \ge 0$. By squaring \eqref{eq:q1Inequality} we have $\J(w)>\J(1)$ if 
\begin{align*}
    q_2>0, \, \, \, q_2:= 32u(a-w)^2-q_1^2.
\end{align*}
Define the following quantity
\begin{align*}
    q_3 = q_2 - 4\alpha(4-\alpha)(w - 1)^2q_1.
\end{align*} Because for $\alpha \in (1,4)$ the second term appearing in $q_3$ is always non-negative, to show $q_2>0$ it is enough to show that $q_3>0$.  
Now $q_3$ admits the factorization
\begin{align*}
    q_3=&\ 4(w - 1)^2 q_4, \\
    q_4\coloneqq&\  (\alpha^4 - 8\alpha^3 + 20\alpha^2 + 8) w^2 + (- 2\alpha^4 + 8\alpha^3 -32\alpha^2 - 16\alpha)w + \\&+(5\alpha^4 - 8\alpha^3 + 24\alpha^2),
\end{align*}
where $q_4$ as a quadratic in $w$ has discriminant
\begin{align*}
\Delta(q_4)&= -16\alpha^2(\alpha-4)^2(\alpha-1)^2(\alpha^2+2).
\end{align*}
It is clear that $\Delta(q_4)<0$ for $\alpha\in(1,4)$. Also, for $\alpha \in (1,4)$ the leading coefficient of $q_4$ is positive:\begin{align*}
    \alpha^4 - 8\alpha^3 + 20\alpha^2 + 8 
    = \alpha^4 - 8\alpha^3 + 16\alpha^2 + 4\alpha^2 + 8 
    = \alpha^2(\alpha - 4)^2 + 4\alpha^2 + 8 > 0.
\end{align*}
Therefore, for every $w\ne 1$ we have $q_4 > 0$, hence $q_3 > 0$, thus $q_2>0$. That is $\J(w) >\J(1)$ for $w \neq 1$. 
\end{proof}

\section{Isosceles Triangles} \label{section:two-equal-angles}

In this section we will prove part (2) of Theorem \ref{thm:theorem}, which is the case 
$$60^\circ \le \varphi_{AB}<\varphi_{AC}=\varphi_{BC}.$$
When $\varphi_{AC}=\varphi_{BC} = 90^\circ$, we are done by Proposition \ref{prp:big-triangle} together with $\J_{ABC}(A) = \J_{ABC}(B) = \sin \varphi_{AB} + 1 < 2 = \J_{ABC}(C)$. We will thus assume in the rest of this section that $\varphi_{AC}=\varphi_{BC} < 90^\circ$. For the same reason as in the beginning of \S \, \ref{section:equiangular}, we can choose representatives such that $a^\top b, \,  a^\top c, \, b^\top c> 0$.  

We may assume without loss of generality that 
\begin{align*}
    a &= (\cos\alpha, \sin\alpha, 0)\\
    b &= (\cos\alpha, -\sin\alpha, 0)\\
    c &= (\cos\beta, 0, \sin\beta),
\end{align*}
where $$\beta = \arccos(\cos\varphi_{AC}/\cos\alpha)\in(0,\pi/2]$$ is the angle from $c$ to the plane $\mathcal{H} = \mathcal{H}_{ab}= \vspan(a,b)$. Set
$$\O_{xy}^+=\{(p_x,p_y,p_z) \in \Sp^2 \mid p_x\ge 0,p_y\ge 0\}.$$

We first prove:

\begin{lemma} 
Suppose that $\mathscr{P}_{abc}\cap \O_{xy}^+=\{a\}$. Then Theorem \ref{thm:theorem}(2) is true. 
\end{lemma}
\begin{proof}

We introduce an additional notation:
$$\O_{x}^+=\{(p_x,p_y,p_z) \in \Sp^2 \mid p_x\ge 0\}$$

Let $M=\operatorname{diag}(1,-1,1)\in\operatorname{SO}(3)$ be the reflection matrix with respect to the $xz$-plane. Since $Ma = b$ and $Mc = c$, we have that $\J_{abc}(Mv) = \J_{abc}(v)$ for any $v \in \Sp^2$. Since $\O_x^+=\O_{xy}^+\cup M\O_{xy}^+$ and $\mathscr{P}_{abc}\cap M\O_{xy}^+= M(\mathscr{P}_{abc}\cap\O_{xy}^+)$, we have
\begin{align*}
    \mathcal{P}_{abc}\cap \O_x^+ = (\mathcal{P}_{abc}\cap \O_{xy}^+) \cup M(\mathcal{P}_{abc}\cap \O_{xy}^+) = \{a\} \cup M\{a\} = \{a,b\}, 
\end{align*}
which concludes the proof.
\end{proof}

Thus in the rest of this section we will show that $\mathscr{P}_{abc}\cap \O_{xy}^+=\{a\}$. We need the following fact:

\begin{lemma} \label{lem:range-of-p}
Let $p=(p_x,p_y,p_z)\in\mathscr{P}_{abc}\cap\O_{xy}^+$. Then $p_z, \, a^\top p,\, b^\top p,\, c^\top p \ge 0$.
\end{lemma}
\begin{proof}
    By hypothesis $p_x, \, p_y\ge 0$ and so $a^\top p=p_x\cos\alpha + p_y\sin\alpha \ge 0$. If $a^\top p=0$, then $p_x=p_y=0$ and so $\J_{abc}(p)=1+1+\dS(c,p)>\J_{abc}(a)$. Hence $a^\top p>0$ and Lemma \ref{lem:non-positive-non-negative} gives $b^\top p, \, c^\top p\ge 0$. Now $e_z=(0,0,1)$ is normal to $\mathcal{H}_{ab}$, thus by Lemma \ref{lem:inner-products-of-normal} we have $p_z \sin\beta=(p^\top e_z)(c^\top e_z) \ge 0$, which means $p_z\ge 0$.
\end{proof}

In what follows, we take $p\in\mathscr{P}_{abc}\cap\O_{xy}^+$ with $p \neq a$ and proceed in several steps to derive a contradiction by showing that $\J_{abc}(p)>\J_{abc}(a)$.

{\bf{Step 1.}} Define the angle 
$$\beta^* = \arccos(\cos(\varphi_{AB})/\cos\alpha) < \arccos(\cos\varphi_{AC}/\cos\alpha) = \beta,$$
and the point 
$$ c^* = (\cos\beta^*,0,\sin\beta^*).$$ Let $C^*$ be the point in $\mathbb{P}^2$ represented by $c^*$. One immediately checks that $\varphi_{AB} = \varphi_{AC^*} = \varphi_{BC^*}$. Geometrically, the point $c^*$ is obtained as follows:

\begin{construction} \label{construction:c*}
Let $\V_{ab}$ be the plane consisting of all points that have equal angles from $a,b$. By hypothesis $c \in \V_{ab} \cap \Sp^2$ and so is the point $e_x=(1,0,0)$ that cuts in half the arc $\Arc{ab}$ and whose angle from $a,b$ is $\varphi_{AB}/2$. Since $\varphi_{AC} = \varphi_{BC} > \varphi_{AB}/2$, $c^*$ is the unique point on the arc $\Arc{ce_x}$ that has angle from $a,b$ equal to $\varphi_{AB}$.
\end{construction}

{\bf{Step 2.}} We define an auxiliary point $c_p$. We start with the following property:

\begin{lemma} \label{lem:lemma}
The point $p$ lies in the spherical triangle of $\Sp^2$ defined by $a,c,e_x$. 
\end{lemma}
\begin{proof}
Any two vertices of a spherical triangle, say $v_1$ and $v_2$, together with the origin, determine a plane $\mathcal{H}_{v_1v_2}$ which cuts the space into two half-spaces. A point $q \in \Sp^2$ lies in the spherical triangle if and only if for any two vertices $v_1$ and $v_2$, the third vertex $v_3$ and $q$ lie in the same half-space of $\mathcal{H}_{v_1v_2}$.

By Lemma \ref{lem:inner-products-of-normal} we have that $b$ and $p$ lie in the same half-space of $\mathcal{H}_{ac}$. Similarly, $c$ and $p$ lie in the same half-space of $\mathcal{H}_{ab}=\mathcal{H}_{ae_x}$. Moreover, $e_y=(0,1,0)$ is normal to $\mathcal{H}_{ce_x}$ and since $p \in \mathcal{O}_{xy}^+$, we have $p^\top e_y=p_y\ge 0$. Also, $a^\top e_y=\sin\alpha\ge 0$ and so $p$ and $a$ lie in the same half-space of $\mathcal{H}_{ce_x}$. 
\end{proof}

Our auxiliary point $c_p$ is defined as follows:
\begin{construction} \label{construction:cp}
By Lemma \ref{lem:lemma} the point $p$ lies in the spherical triangle defined by $a,c,e_x$. Extend the arc $\Arc{ap}$ in the direction from $a$ to $p$ so that it intersects the opposite arc $\Arc{ce_x}$ of the spherical triangle at a point $c'$. From Construction \ref{construction:c*} we have that $c^* \in \Arc{ce_x}$. If $c' \in \Arc{c^*e_x}$, we let $c_p=c^*$. Otherwise, we let $c_p=c'$.
\end{construction}

We next describe the coordinates of $c_p$. Note that the arc $\Arc{e_xc}$ can be written parametrically as
$$g(\psi) = (\cos \psi, 0, \sin \psi)$$ with $\psi \in [0,\beta]$. Thus there are unique $\beta', \beta^* \in [0,\beta]$ such that $c'=g(\beta')$ and $c^*=g(\beta^*)$. With $\beta_p = \max(\beta',\beta^*)$ we have   
$$c_p = (\cos\beta_p,0,\sin\beta_p).$$
In what follows, $C_p$ will be the point in $\mathbb{P}^2$ represented by $c_p$.

{\bf{Step 3.}} In step 4 we will prove the inequalities
\begin{align}
    \J_{abc}(p)-\J_{abc_p}(p) &\ge \J_{abc}(a)-\J_{abc_p}(a), \label{ineq:greater-difference}\\
    \J_{abc_p}(p)&\ge \J_{abc_p}(a) \label{ineq:Jp-greater-than-Ja},
\end{align} with at least one of them being strict. This will establish that $\J_{abc}(p) > \J_{abc}(a)$ and thus conclude this section. Here we prove certain auxiliary facts that we need. 

\begin{lemma} \label{lem:range-of-slope}
    Write $p=(p_x,p_y,p_z)$, then $p_x\sin\alpha-p_y\cos\alpha>0$.
\end{lemma}
\begin{proof}
From Lemma \ref{lem:range-of-p} we know that $b^\top p\ge 0$. With $n=a\times c$ Lemma \ref{lem:inner-products-of-normal} gives 
$(b^\top n)(n^\top p) \ge 0$. Since $b^\top n=\frac{1}{2}\sin\beta\sin 2\alpha>0$, we thus have $n^\top p\ge 0$. In terms of coordinates this reads
    $$ p_x\sin\alpha\sin\beta - p_y\cos\alpha\sin\beta - p_z\sin\alpha\cos\beta\ge 0.$$
Again from Lemma \ref{lem:range-of-p} we know that $p_z \ge 0$ and so  
    $$\sin\beta(p_x\sin\alpha-p_y\cos\alpha) \ge p_z\sin\alpha\cos\beta\ge 0.$$
    We conclude that $p_x\sin\alpha-p_y\cos\alpha\ge 0$. If $p_x\sin\alpha-p_y\cos\alpha=0$, then the above inequality gives $p_z=0$. But this would imply that $p = a$, a contradiction. 
    \end{proof}

\begin{lemma} \label{lem:function-f}
    Define function $f: [0, \pi/2] \rightarrow \mathbb{R}$ as
    $$f(\psi) = (p_x\sin\alpha - p_y\cos\alpha)\sin\psi - p_z\sin\alpha\cos\psi.$$
    Then $f(\beta')=0$ and $f$ is strictly ascending on $[0,\pi/2]$.
\end{lemma}

\begin{proof}
By construction $a,c',p$ are coplanar and so $(a\times c')^\top p=0$, which in coordinates is the same as $f(\beta')=0$. From Lemma \ref{lem:range-of-slope} we have $p_x\sin\alpha-p_y\cos\alpha>0$, while from Lemma \ref{lem:range-of-p} we have $p_z\ge 0$, thus $f$ is strictly ascending.
\end{proof}

\begin{lemma} \label{lem:greater-inner-product}
    For $\psi\in[\beta',\beta]$ we have $g(\psi)^\top p \ge g(\psi)^\top a\ge 0$.
\end{lemma}

\begin{proof}
Note that $g(\psi)^\top a=\cos\psi\cos\alpha\ge 0$ for $\psi\in[\beta',\beta]$, which proves the second inequality in the statement.

For the first inequality, we first prove it for $\psi = \beta, \beta'$. Lemma \ref{lem:range-of-p} gives $c^\top p\ge 0$, while Lemma \ref{lem:distance-to-vertices} gives $c^\top p \ge c^\top a$. Recalling that $c=g(\beta)$, we have  $g(\beta)^\top(p-a)\ge 0$. By the construction of $c'$, we have $p \in \Arc{ac'}$ and so $\operatorname{length}\Arc{c'p} \le \operatorname{length} \Arc{ac'}$. Hence $c'^\top p \ge c'^\top a\ge 0$ and so $g(\beta')^\top (p-a)\ge 0$, where we recall that $c' = g(\beta')$. 

We next consider $\psi \in (\beta', \beta)$. Define function $h: [0, \beta-\beta'] \rightarrow \R$ as $$h(\delta) = g(\delta+\beta')^\top(p-a).$$ It is enough show that $h(\delta)\ge 0$ for any $\delta\in(0,\beta-\beta')$. A direct calculation gives $$h(\delta)= \cos\delta\left(g(\beta')^\top (p-a)\right) + \sin\delta\left(g(\beta'+\pi/2)^\top (p-a)\right).$$
We already have $g(\beta')^\top(p-a)\ge 0$. If $g(\beta'+\pi/2)^\top (p-a)\ge 0$, then $h(\delta)\ge 0$. Otherwise, $h(\delta)$ is decreasing and so $h(\delta)\ge h(\beta-\beta')=g(\beta)^\top(p-a)\ge 0$.
\end{proof}

For convenience, we will be using Newton notation $(g(\psi)^\top p)'$ and $(g(\psi)^\top a)'$ for the derivatives
$$\frac{d\left(g(\psi)^\top p\right)}{d\psi} = -p_x\sin\psi+p_z\cos\psi,  \quad \frac{d\left(g(\psi)^\top a\right)}{d\psi}=-\cos\alpha\sin\psi.$$

\begin{lemma} \label{lem:decending-inner-products}
    For $\psi\in(\beta',\beta)$, we have $(g(\psi)^\top p)'<0$, $(g(\psi)^\top a)'<0$, and
    \begin{align*}
        -\sin\alpha(g(\psi)^\top p)' > -p_y(g(\psi)^\top a)' \ge 0. 
    \end{align*}
\end{lemma}
\begin{proof}
By immediate inspection $(g(\psi)^\top a)' = -\cos\alpha\sin\psi < 0$. From Lemma \ref{lem:function-f}, $f(\beta')=0$ and $f$ is strictly ascending on $\psi\in[\beta',\beta]$. Then $f(\psi)>f(\beta')=0$ for $\psi\in(\beta',\beta)$. In terms of coordinates this reads
$$p_x\sin\alpha\sin\psi - p_y\cos\alpha\sin\psi - p_z\sin\alpha\cos\psi > 0,$$ which can be equivalently written as 
$$ -\sin\alpha(g(\psi)^\top p)' > -p_y(g(\psi)^\top a)' \ge 0,$$ and the statement follows.
\end{proof}

\begin{lemma} \label{lem:greater-inner-product-of-tangents}
For every $\psi\in(\beta',\beta)$ we have 
    \begin{align} \label{ineq:greater-inner-product-of-tangents}
        \frac{-(g(\psi)^\top p)'}{\sqrt{1-\left(g(\psi)^\top p\right)^2}} >
        \frac{-(g(\psi)^\top a)'}{\sqrt{1-\left(g(\psi)^\top a\right)^2}} > 0.
    \end{align}
\end{lemma}

\begin{proof}
    We first show that both sides of \eqref{ineq:greater-inner-product-of-tangents} are well-defined. By Lemma \ref{lem:decending-inner-products}, both $g(\psi)^\top p$ and $g(\psi)^\top a$ are strictly descending on $(\beta',\beta)$. Since these are continuous functions, they are also strictly descending on $[\beta',\beta)$. For $\psi\in(\beta',\beta)$ 
    \begin{align*}
        g(\psi)^\top p<g(\beta')^\top p \le 1,\qquad g(\psi)^\top a<g(\beta')^\top a \le 1.
    \end{align*} One readily verifies that $g(\psi)^\top a\ge 0$, and similarly $g(\psi)^\top p\ge 0$ by Lemma \ref{lem:range-of-p} and the hypothesis that $p \in \mathcal{O}_{xy}^+$. Thus the expressions in \eqref{ineq:greater-inner-product-of-tangents} are well-defined. The second inequality in \eqref{ineq:greater-inner-product-of-tangents} follows from the assertion $(g(\psi)^\top a)'<0$ of Lemma \ref{lem:decending-inner-products}. In view of that, the first inequality is equivalent to
    \begin{align} \label{ineq:lem20-2}
        \frac{\left((g(\psi)^\top p)'\right)^2}{1-\left(g(\psi)^\top p\right)^2} 
        >\frac{\left((g(\psi)^\top a)'\right)^2}{1-\left(g(\psi)^\top a\right)^2}. 
    \end{align} A simple calculation gives
    \begin{align*}
        1-\left(g(\psi)^\top p\right)^2 &= p_y^2 + \left((g(\psi)^\top p)'\right)^2, \\
        1-\left(g(\psi)^\top a\right)^2 & = \sin^2\alpha + \left((g(\psi)^\top a)'\right)^2.
    \end{align*} Substituting into (\ref{ineq:lem20-2}) we get the equivalent expression
    \begin{alignat*}{2}
         \sin^2\alpha\left((g(\psi)^\top p)'\right)^2 &> p_y^2\left((g(\psi)^\top a)'\right)^2, 
    \end{alignat*} which is true from Lemma \ref{lem:decending-inner-products}.
\end{proof}

\begin{lemma} \label{lem:greater-derivatives}
    For any $\psi \in (\beta', \beta)$ we have
    \begin{align} \label{ineq:greater-derivatives}
        \frac{d\sqrt{1-\left(g(\psi)^\top p\right)^2}}{d\psi}>
        \frac{d\sqrt{1-\left(g(\psi)^\top a\right)^2}}{d\psi}
    \end{align}
\end{lemma}
\begin{proof}
By Lemmas \ref{lem:greater-inner-product} and \ref{lem:greater-inner-product-of-tangents}
$- \big(g(\psi)^\top p \big)' > - \big(g(\psi)^\top a \big)>0.$ By Lemma \ref{lem:greater-inner-product} 
$$\frac{g(\psi)^\top p}{\sqrt{1-\big(g(\psi)^\top p\big)^2}}  > \frac{g(\psi)^\top a}{\sqrt{1-\big(g(\psi)^\top a\big)^2}}.$$ 
Multiplying these two inequalities gives the inequality in the statement.
\end{proof}

{\bf{Step 4.}} We are now ready to prove inequalities \eqref{ineq:greater-difference} and \eqref{ineq:Jp-greater-than-Ja}.
The following lemma proves \eqref{ineq:greater-difference}. 

\begin{lemma} \label{lem:greater-difference}
We have $\J_{abc}(p)-\J_{abc_p}(p) \ge \J_{abc}(a)-\J_{abc_p}(a)$ with equality if and only if $\beta_p=\beta$.
\end{lemma}

\begin{proof}
    If $\beta_p=\beta$, then $c=c_p$ and the equality is trivial. Otherwise, we have the equivalences 
    \begin{align*}
        &&\J_{abc}(p) - \J_{abc_p}(p) & > \J_{abc}(a) - \J_{abc_p}(a) \\
        \LRarrow && \dS(c,p) - \dS(c_p,p) &> \dS(c,a) - \dS(c_p,a) \\
        \LRarrow && \dS\big(g(\beta),p\big) - \dS\big(g(\beta_p),p\big) &> \dS\big(g(\beta),a\big) - \dS\big(g(\beta_p),a\big) \\
        \LRarrow &&
        \int_{\beta_p}^{\beta} \frac{d\sqrt{1-\big(g(\psi)^\top p\big)^2}}{d\psi}d\psi &> 
        \int_{\beta_p}^{\beta} \frac{d\sqrt{1-\big(g(\psi)^\top a\big)^2}}{d\psi}d\psi,
    \end{align*} the last of them true in view of Lemma \ref{lem:greater-derivatives}, continuity and the fact that $\beta_p \ge \beta'$.
\end{proof}

For \ref{ineq:Jp-greater-than-Ja} we first need a lemma:

\begin{lemma} \label{lem:bound-of-phi}
Let $\chi$ be the angle of $b$ from $\mathcal{H}=\vspan(a,c_p)$. Then $\cos^2\chi < \frac{1}{2}$.
\end{lemma}
\begin{proof}
Set $$d = \frac{c_p-aa^\top c_p}{||c_p-aa^\top c_p||}=\frac{c_p-aa^\top c_p}{\sin\varphi_{AC_p}},$$ and note that $\{a,d\}$ is an orthonormal basis of $\mathcal{H}$. Then the orthogonal projection of $b$ onto $\mathcal{H}$ is $b_\mathcal{H} = (a^\top b)a + (d^\top b)d$. Now
    $$\cos\chi= b^\top \frac{b_\mathcal{H}}{\|b_\mathcal{H}\|_2} = \frac{(a^\top b)^2+(d^\top b)^2}{\sqrt{(a^\top b)^2+(d^\top b)^2}}
    =\sqrt{(a^\top b)^2+(d^\top b)^2}.$$ Recalling that $60^\circ \le \varphi_{AB}\le \varphi_{AC_p}=\varphi_{BC_p}$, we have
    \begin{align*}
        (d^\top b)^2 &= \frac{(\cos\varphi_{BC_p}-\cos\varphi_{AB}\cos\varphi_{AC_p})^2}{\sin^2\varphi_{AC_p}}= \frac{\cos^2\varphi_{AC_p}(1-\cos\varphi_{AB})^2}{\sin^2\varphi_{AC_p}} \\
        &\le \cos^2\varphi_{AB} \frac{(1-\cos\varphi_{AB})^2}{\sin^2\varphi_{AB}} < \cos^2\varphi_{AB} \frac{(1-\cos\varphi_{AB})(1+\cos\varphi_{AB})}{\sin^2\varphi_{AB}} \\
        &= \cos^2\varphi_{AB} \le \frac{1}{4}. 
    \end{align*}
Now $\cos^2\chi = (a^\top b)^2 + (d^\top b)^2 < \cos^2\varphi_{AB} + \frac{1}{4} \le\frac{1}{2}$. 
\end{proof}

Finally, we prove inequality \eqref{ineq:Jp-greater-than-Ja}, which concludes this section:

\begin{lemma} \label{lem:Jp-greater-than-Ja}
We have $\J_{abc_p}(p)\ge \J_{abc_p}(a)$ with strict inequality if $\beta_p > \beta^*$.
\end{lemma}

\begin{proof}
If $\beta_p = \beta^*$, then $c_p=c^*$ and $a,b,c_p$ are equiangular by Construction \ref{construction:c*}. Since $\varphi_{AB} \ge 60^\circ$, Proposition \ref{prop:equiangular} gives $\J_{abc_p}(p)\ge \J_{abc_p}(a)$. So suppose that 
$\beta_p > \beta^*$. In that case we have $c_p=c'$, $\beta_p\in(\beta^*,\beta]$ and $a,c_p,p$ are coplanar. Hence 
    $$\cos\varphi_{AC_p}=a^\top c_p = \cos\alpha\cos\beta_p \in [\cos\alpha\cos\beta,\cos\alpha\cos\beta^*) = [\cos\varphi_{AC},\cos\varphi_{AB}),$$ and so $\varphi_{AC_p}\in(\varphi_{AB},\varphi_{AC}]$. We need a new coordinate system which brings $a,c_p$ and $p$ in the $xOy$ plane as follows: 
    \begin{align*}
        a &= (1, 0, 0)\\
        b &= (\cos\chi\cos\theta,\cos\chi\sin\theta,\sin\chi)\\
        c_p &= (\cos\varphi_{AC_p},\sin\varphi_{AC_p}, 0)\\
        p &= (\cos\varphi_{AP},\sin\varphi_{AP}, 0).
    \end{align*} Here $\chi$ is the angle of $b$ from the plane $\vspan(a,c_p)$ and $\theta\in(-\pi,\pi]$. Then
    $$\J_{abc}(p) =  \sin\varphi_{AP} + \sin(\varphi_{AC_p}-\varphi_{AP}) + \sqrt{1-\cos^2\chi\cos^2(\varphi_{AP}-\theta)}.$$
    Consider the function
    $$\J(\delta) = \sin\delta + \sin(\varphi_{AC_p}-\delta) + \sqrt{1-\cos^2\chi\cos^2(\delta-\theta)}.$$
    We have 
    \begin{align*}
    \J_{abc}(a) &= \J(0) =\sin\varphi_{AC_p} + \sqrt{1-(a^\top b)^2} = \sin\varphi_{AC_p}+\sin\varphi_{AB}, \\
    \J_{abc}(p) &= \J(\varphi_{AC_p}) =\sin\varphi_{AC_p} + \sqrt{1-(c_p^\top b)^2} = 2\sin\varphi_{AC_p},
    \end{align*} where we used $\varphi_{AC_p} = \varphi_{BC_p}$. Since $p \in \Arc{ac_p}$, we have $0\le \varphi_{AP}\le \varphi_{AC_p}$. By Lemma \ref{lem:bound-of-phi}, $\cos^2\chi \le 1/2$. Consider the second derivative of $\J(\delta)$ for $\delta\in[0,\varphi_{AC_p}]$: 
        \begin{align*}
        \J''(\delta)=& - \sin\delta - \sin(\varphi_{AC_p}-\delta)
        -\cos^2\chi\sin^2(\theta-\delta)\big(1-\cos^2\chi\cos^2(\theta-\delta)\big)^{-3/2}\\
        &+\big(1-\cos^2\chi\cos^2(\theta-\delta)\big)^{-1/2}\cos^2\chi\cos^2(\theta-\delta)\\
        \le &
        - \sin\delta - \sin(\varphi_{AC_p}-\delta)
        +\big(1-\cos^2\chi\cos^2(\theta-\delta)\big)^{-1/2}\cos^2\chi\cos^2(\theta-\delta)\\
        \le &
        - \sin\varphi_{AC_p}
        +(1-\cos^2\chi)^{-1/2}\cos^2\chi\\
        \le& - \frac{\sqrt{3}}{2} + \frac{\sqrt{2}}{2} < 0
    \end{align*}
    We see that $\J(\delta)$ is strictly concave on $[0,\varphi_{AC_p}]$. Then we have 
    $$\J(\varphi_{AP})\ge\min\{\J(0),\J(\varphi_{AC_p})\}=\J(0),$$
    and the equality is achieved if and only if $\varphi_{AP} = 0$, which means $p=a$. Hence, if $p\ne a$, we get 
    $\J(\delta)>\J(0)$, that is $\J_{abc_p}(p)>\J_{abc_p}(a).$
\end{proof}

\section{General Triangles} \label{section:general-case}

In this section we complete the proof of Theorem \ref{thm:theorem} by proving part (1), which is concerned with the general configuration 
$$60^\circ \le \varphi_{AB} \le \varphi_{AC} < \varphi_{BC}.$$ 
It is immediate that $\J_{ABC}(A)<\J_{ABC}(B) \le \J_{ABC}(C)$ and thus in the case of a big triangle $\mathscr{P}_{ABC} = \{A\}$ by Proposition \ref{prp:big-triangle}. We may thus assume that the projective triangle $\triangle ABC$ is not big. Moreover, as in the beginning of \S \, \ref{section:equiangular}, we may assume representatives $a,b,c$ such that $a^\top b, \, a^\top c, \, b^\top c >0$. Without loss of generality, we assume coordinates
\begin{align*}
    a&=(0,0,1)\\
    b&=(\sin\varphi_{AB},0,\cos\varphi_{AB})\\
    c&=(\sin\varphi_{AC}\cos\alpha,\sin\varphi_{AC}\sin\alpha,\cos\varphi_{AC}), 
\end{align*}
where $\alpha \in (0,\pi]$ is given by 
\begin{align*}
    \alpha=\arccos\frac{\cos\varphi_{BC}-\cos\varphi_{AB}\cos\varphi_{AC}}{\sin\varphi_{AB}\sin\varphi_{AC}}. 
\end{align*}

Let $P\in\mathscr{P}_{ABC}$ and $p \in \Sp^2$ a representative of $P$ such that $a^\top p\ge 0$. We assume $P \neq A$, write $p=(\sin\omega\cos\theta,\sin\omega\sin\theta,\cos\omega)$ with $\omega\in(0,\pi/2]$ and $\theta\in(-\pi,\pi]$, and we exhibit a contradiction. First we determine the range of $\theta$:

\begin{lemma} \label{lem:range-theta}
We have $0 \le \theta \le \alpha$. 
\end{lemma}
\begin{proof}
Since by hypothesis $a^\top p\ge 0$, Lemma \ref{lem:non-positive-non-negative} gives $b^\top p, \, c^\top p\ge 0$. Let $n_{ab}=a\times b$ be the normal vector to the plane spanned by $a,b$. Writing $n_{ab}$ in coordinates gives $$n_{ab}^\top c=\sin\varphi_{AB}\sin\varphi_{AC}\sin\alpha > 0.$$ This, together with $c^\top p\ge 0$ and Lemma \ref{lem:inner-products-of-normal}, give $$n_{ab}^\top p = \sin\varphi_{AB}\sin\omega\sin\theta\ge 0.$$  This implies $\sin\theta\ge 0$ and so $\theta \ge 0$. A similar argument with $n_{ca} = c\times a$ gives $\sin(\alpha-\theta)\ge 0$, that is $\alpha-\theta \ge 0$. 
\end{proof}

Next, we define two auxiliary points:

\begin{construction} \label{construction:b',c'}
With angle $\alpha' \in [0, \pi]$ given by 
$$ \alpha'=\arccos\frac{\cos\varphi_{AC}-\cos\varphi_{AB}\cos\varphi_{AC}}{\sin\varphi_{AB}\sin\varphi_{AC}},$$
we define points $b',c' \in \Sp^2$, respectively representing points $B', C' \in \P^2$, as 
\begin{align*}
    b'&=\big(\sin\varphi_{AB}\cos(\alpha-\alpha'),\sin\varphi_{AB}\sin(\alpha-\alpha'),\cos\varphi_{AB}\big),\\
    c'&=(\sin\varphi_{AC}\cos\alpha',\sin\varphi_{AC}\sin\alpha',\cos\varphi_{AC}).
\end{align*} One verifies easily that $a^\top c' = b^\top c'= b'^\top c = \cos\varphi_{AC}$ and $a^\top b'=\cos\varphi_{AB}$.
\end{construction}

We also need:

\begin{lemma} \label{lem:alpha'}
With $\alpha'$ as in Construction \ref{construction:b',c'} we have $\alpha' \in (\alpha/2,\alpha)$. 
\end{lemma}
\begin{proof}
By hypothesis $\varphi_{AC}<\varphi_{BC}$ and so by definition $\alpha'<\alpha$. We are going to prove $2\alpha' > \alpha$.
First we have $\alpha < 2\pi/3$, since $\varphi_{AB}, \varphi_{AC} \ge 60^\circ$ and so 
$$\cos\alpha > \frac{-\cos\varphi_{AB}\cos\varphi_{AC}}{\sin\varphi_{AB}\sin\varphi_{AC}}=-\cot \varphi_{AB} \cot \varphi_{AC}> - \frac{1}{3} > - \frac{1}{2}.$$
Moreover, since  
$$ \frac{\cos\varphi_{AC}}{\sin\varphi_{AC}}\le \frac{\cos\varphi_{AB}}{\sin\varphi_{AB}}<\frac{1+\cos\varphi_{AB}}{2\sin\varphi_{AB}},$$ we get $\alpha' > \pi/3$ from 
$$\cos\alpha' = \frac{\cos\varphi_{AC}}{\sin\varphi_{AC}} \frac{1-\cos\varphi_{AB}}{\sin\varphi_{AB}} < \frac{1+\cos\varphi_{AB}}{2\sin\varphi_{AB}}\frac{1-\cos\varphi_{AB}}{\sin\varphi_{AB}}
    = \frac{1}{2}.$$ Thus $2 \alpha' > 2 \pi / 3 > \alpha$ from which we read $\alpha' > \alpha / 2$.
\end{proof}

Let us consider the arrangements $ABC'$ and $AB'C$. We have:
\begin{lemma} \label{lem:ABC'-AB'C}
$\{A,B\} \subset \mathscr{P}_{ABC'}$ and $\{A,B'\} \subset \mathscr{P}_{AB'C}$. 
\end{lemma}
\begin{proof}
By Construction \ref{construction:b',c'} we have 
\begin{align*}
A,B,C':& \, \varphi_{AB} \le  \varphi_{AC'} = \varphi_{BC'}, \\
A,B',C:& \, \varphi_{AB'} \le  \varphi_{AC} = \varphi_{B'C}. 
\end{align*} Hence the arrangements $ABC'$ and $AB'C$ fall either in part (2) or part (3) of Theorem \ref{thm:theorem}, which have already been proved. 
\end{proof}

Now, from Lemma \ref{lem:ABC'-AB'C} we have
\begin{align*}
\J_{abc'}(p) \ge & \J_{abc'}(a), \\ 
 \J_{ab'c}(p)\ge & \J_{ab'c}(a).
\end{align*} From Lemma \ref{lem:range-theta} and Lemma \ref{lem:alpha'} we have that either $\theta\in[0,\alpha']$ or $\theta\in[\alpha-\alpha',\alpha]$. If $\theta\in[0,\alpha']$, and since $\alpha \in (0, \pi]$, we have $0 \le \alpha'-\theta < \alpha - \theta$ and so 
\begin{align*}
    0\le c^\top p &= \sin\varphi_{AC}\sin\omega\cos(\alpha-\theta)+\cos(\omega) \cos(\varphi_{AC}) \\ &< \sin\varphi_{AC}\sin\omega\cos(\alpha'-\theta) +\cos(\omega) \cos(\varphi_{AC})= c'^\top p.
\end{align*}
Hence $\dS(c,p)>\dS(c',p)$ and we get the contradiction $$\J_{abc}(p)>\J_{abc'}(p)\ge \J_{abc'}(a)=\J_{abc}(a).$$ 
By a similar argument, if $\theta \in [\alpha-\alpha',\alpha]$ we get $0\le b^\top p  < b'^\top p$ and thus arrive at the contradiction $\J_{abc}(p)>\J_{ab'c}(p)\ge \J_{ab'c}(a)=\J_{abc}(a)$, concluding the proof of Theorem \ref{thm:theorem}.

\bibliographystyle{amsalpha}
\bibliography{TsakirisXu-22.bbl}


\end{document}